\documentclass[a4paper,12pt]{elsarticle}
\usepackage{lineno,hyperref}
\modulolinenumbers[5]

\usepackage{tikz}
\usepackage{mathrsfs}
\usepackage{booktabs,caption,fixltx2e}
\usepackage[flushleft]{threeparttable}
\usepackage{amsfonts}
\usepackage{amsmath}
\usepackage{amsthm}
\usepackage[all,cmtip]{xy}
\usepackage{amsthm}   \usepackage{amsfonts}     \usepackage{amssymb}
\usepackage{amsmath}  \usepackage{mathrsfs}     \usepackage{mathtools}

\newcommand{\CM}{{\mathrm{CM}}}

\newcommand{\Ker}{{\mathrm{Ker}\,}}

\newcommand{\Core}{{\mathrm{Core}\,}}
\newcommand{\Fix}{{\mathrm{Fix}\,}}
\newcommand{\Smooth}{{\mathrm{Smooth}\,}}

\newcommand{\Orbit}{{\mathrm{Orbit}}}

\newcommand{\lcm}{{\mathrm{lcm}}}

\newtheorem{theorem}{Theorem}
\newtheorem{lemma}[theorem]{Lemma}
\newtheorem{proposition}[theorem]{Proposition}
\newtheorem{corollary}[theorem]{Corollary}

\theoremstyle{definition}
\newtheorem{definition}[theorem]{Definition}
\newtheorem{example}[theorem]{Example}

\newtheorem{remark}{Remark}
\newtheorem*{case}{Case}

\begin{document}
\begin{frontmatter}
\title{Smooth skew-morphisms of the dihedral groups}



\author[addr1,addr2]{Na-Er Wang}
\ead{wangnaer@zjou.edu.cn}
\author[addr1,addr2]{Kan Hu\corref{mycorrespondingauthor}}
\cortext[mycorrespondingauthor]{Corresponding author}
\ead{hukan@zjou.edu.cn}
\author[addr3]{Kai Yuan}
\ead{pktide@163.com}
\author[addr4]{Jun-Yang Zhang}
\ead{jyzhang@cqnu.edu.cn}
\address[addr1]{School of Mathematics, Physics and Information Science, Zhejiang Ocean University, Zhoushan, Zhejiang 316022, People's Republic of China}
\address[addr2]{Key Laboratory of Oceanographic Big Data Mining \& Application of Zhejiang Province, Zhoushan, Zhejiang 316022, People's Republic of China}
\address[addr3]{School of Mathematics, Capital Normal University, Beijing 100037, People's Republic of China}
\address[addr4]{School of Mathematical Sciences, Chongqing Normal University, Chongqing 401331, People's Republic of China}
\begin{abstract}
A skew-morphism $\varphi$ of a finite group $A$ is a permutation on $A$ such that $\varphi(1)=1$ and $\varphi(xy)=\varphi(x)\varphi^{\pi(x)}(y)$ for all
$x,y\in A$ where $\pi:A\to\mathbb{Z}_{|\varphi|}$ is an integer function. A skew-morphism is smooth if $\pi(\varphi(x))=\pi(x)$ for all $x\in A$. The concept of smooth skew-morphisms is a generalization of that of $t$-balanced skew-morphisms. The aim of the paper is to develop a general theory of  smooth skew-morphisms. As an application we classify smooth skew-morphisms of the dihedral groups.
\end{abstract}

\end{frontmatter}

\section{Introduction}
A \textit{skew-morphism} of a finite group $A$ is a permutation $\varphi$ of order $n$ on the underlying set of $A$
fixing the identity element of $A$, and for which there exists an integer function $\pi:A\to\mathbb{Z}_n$ such that $\varphi(xy)=\varphi(x)\varphi^{\pi(x)}(y)$
for all $x,y\in A$. The function $\pi$ is called the \textit{power function} associated with $\varphi$. The concept of skew-morphisms was first introduced
by Jajcay and \v{S}ir\'a\v{n} as an important tool to investigate regular Cayley maps~\cite{JS2002}. It has been shown that skew-morphisms are also closely related to group factorisations with a cyclic complement~\cite[Proposition 3.1]{CJT2016}. Thus the study of skew-morphisms is important for both combinatorics and algebra.

 Let $A$ be a finite group, $X$ a generating set of $A$ and $P$ a cyclic permutation of $X$. A \textit{Cayley map} $M=\CM(A,X,P)$ is a $2$-cell embedding of a Cayley graph $G=C(A,X)$ into an orientable surface such that the local
cyclic orientation of the darts $(g,x)$ emanating from any vertex $g$ induced by the orientation of the supporting surface agrees to
the prescribed cyclic permutation $P$ of $X$. The left regular representation
 of the underlying group $A$ of a Cayley map $M=\CM(A,X,P)$ induces a vertex-transitive action of a subgroup of orientation-preserving
 automorphisms of $M$ on the vertices of the map. It follows that $M$ is regular if and only if $M$ admits an automorphism which fixes a vertex,
 say the identity vertex $1$, and maps the dart $(1,x)$ to $(1,P(x))$. It is a non-trivial observation that a Cayley map $\CM(A,X,P)$ is regular if and only if there is a skew-morphism $\varphi$ of $A$ such that the restriction $\varphi\restriction_X$ of $\varphi$ to $X$ is $P$~\cite[Theorem 1]{JS2002}.

Among the variety of problems considered in this direction the most important seems to be the classification of regular Cayley maps for a given family of finite groups. This problem is completely settled for finite cyclic groups~\cite{CT2014}, and only partial results are known for other abelian groups~\cite{CJT2007, CJT2016,Zhang2014}.
For dihedral groups $D_n$ of order $2n$, if $n$ is odd then this problem is solved~\cite{KMM2013},  whereas if $n$ is even only partial classification is at hand~\cite{KK2016,KKF2006,WF2005,Zhang2014,Zhang20152,Zhang2015}.
For other non-abelian groups the interested reader is referred to \cite{KO2008, Oh2009,WF2005}.

Although skew-morphisms are usually investigated along with regular Cayley
maps, they deserve an independent study in a purely algebraic setting. Let $G=AB$ be a finite group factorisation where $A$ and $B$ are subgroups of $G$ with $A\cap B=1$. If $B=\langle b\rangle$ is cyclic then the commuting rule $bx=\varphi(x)b^{\pi(x)}$ for all $x\in A$ determines a skew-morphism $\varphi$ of $A$ with the associated power function $\pi$. Conversely each skew-morphism $\varphi$ of $A$ determines a group factorisation $A_L\langle \varphi\rangle$ with $A_L\cap\langle\varphi\rangle=1$ where $A_L$ denotes the left regular representation of $A$~\cite[Proposition 3.1]{CJT2016}. Thus there is a correspondence between skew-morphisms and group factorisations with cyclic complements.

For a skew-morphism $\varphi$ of $A$ of order $n$, it has been well known that the subgroups $\Ker\varphi=\{x\in A\mid \pi(x)=1\}$ and $\Core\varphi=\bigcap_{i=1}^n\varphi^i(\Ker\varphi)$, called the \textit{kernel} and \textit{core} of $\varphi$ respectively, play important roles in the investigation of skew-morphisms. Based on properties of these subgroups this paper is devoted to the exposition of a general theory
 on skew-morphisms $\varphi$ for which the kernel $\Ker\varphi$ is invariant with respect to $\varphi$, that is, $\varphi(\Ker\varphi)=\Ker\varphi$.
 Particular attention will be paid on a subclass which we call \textit{smooth} skew-morphisms, which means that the associated power function $\pi$ takes constant values on orbits of $\varphi$. Smooth skew-morphisms of the cyclic groups have been recently classified by Bachrat\'y and Jajcay in~\cite{BJ2014, BJ2016}. In this paper employing the theory developed we present a classification of smooth skew-morphisms of the dihedral groups.

\section{Preliminaries}
In this section we summarise some basic results on skew-morphisms which will be used later.

\begin{lemma}\label{BASIC} \cite{JS2002}
  Let $\varphi$ be a skew-morphism of a finite group $A$, and $\pi:A\to \mathbb{Z}_n$ be the associated power function where $n=|\varphi|$. Then the following hold true:
\begin{itemize}
\item[\rm(i)]for any positive integer $k$ and for all $x,y\in A$, $\varphi^k(xy)=\varphi^k(x)\varphi^{\sigma(x,k)}(y)$  where
$\sigma(x,k)=\sum\limits_{i=1}^k\pi(\varphi^{i-1}(x)),$
\item[\rm(ii)]for all $x,y\in A$, $\pi(xy)\equiv\sigma(y,\pi(x))\pmod{n}$,
\item[\rm(iii)] $K:=\Ker\varphi=\{x\in A\mid \pi(x)=1\}$ is a subgroup of $A$,
\item[\rm(iv)]for all $x,y\in A$, $\pi(x)=\pi(y)$ if and only if $Kx=Ky$,
\item[\rm(v)]$\Fix\varphi=\{x\in A\mid \varphi(x)=x\}$ is a $\varphi$-invariant subgroup of $A$.
\end{itemize}
\end{lemma}
The subgroups $\Ker\varphi$ and $\Fix\varphi$ of $A$ will be called the \textit{kernel} and \textit{fixed-point subgroup} of $\varphi$.
\begin{lemma}\cite{ Zhang2015}
  Let $\varphi$ be a skew-morphism of a finite group $A$, and $\pi:A\to \mathbb{Z}_n$ be the associated power function where $n=|\varphi|$.
Then the set $\Core\varphi=\bigcap\limits_{i=1}^n\varphi^i(\Ker\varphi)$ is a $\varphi$-invariant normal subgroup of $A$ contained in $\Ker\varphi$.
\end{lemma}
\begin{lemma}\cite{Hu2012}\label{CONJ}
  Let $\varphi$ be a skew-morphism of a finite group $A$, and $\pi:A\to \mathbb{Z}_n$ be the associated power function where $n=|\varphi|$. Then
for any automorphism $\gamma$ of $A$, $\psi=\gamma^{-1}\varphi\gamma$ is a skew-morphism of $A$
with power function $\pi_\psi=\pi\gamma^{-1}$. In particular $\Ker\psi=\gamma^{-1}(\Ker\varphi)$ and $\Core\psi=\gamma^{-1}(\Core\varphi)$.
\end{lemma}
\begin{lemma}\cite{BJ2014, CJT2016}\label{POWER}
  Let $\varphi$ be a skew-morphism of a finite group $A$, and $\pi:A\to \mathbb{Z}_n$ be the associated power function where $n=|\varphi|$.
Then for any positive integer $k$, $\mu=\varphi^k$ is a skew-morphism of $A$ if and only if the congruences \begin{align}\label{CONG}
 kt\equiv\sigma(x,k)\pmod{n}
 \end{align}
 are solvable for all $x\in A$, where $\sigma(x,k)=\sum\limits_{i=1}^k\pi(\varphi^{i-1}(x)).$  Moreover, if $\mu$ is a skew-morphism of $A$,
then it has order $m=n/\gcd(n,k)$ and for each $x\in A$, $\pi_\mu(x)$
is the solution of Eq.~\eqref{CONG} in $\mathbb{Z}_m$; in particular $\Core\varphi\leq\Core\mu$.
 \end{lemma}

\begin{lemma}\label{LESS}\cite{CJT2016}
  Let $\varphi$ be a skew-morphism of a non-trivial finite group $A$. Then $|\varphi|\leq |A|$ and $\Ker\varphi>1$.
\end{lemma}

\begin{lemma}\cite{JN2014}\label{INVERSE}
 Let $\varphi$ be a skew-morphism of a finite group $A$, then for each $x\in A$,  $O_{x^{-1}}=O_x^{-1}$
  where $O_x^{-1}=\{x^{-1}\mid x\in O_x\}$. \end{lemma}

\begin{lemma}\label{ORBIT1}\cite{Hu2012}
 Let $\varphi$ be a skew-morphism of a finite group $A$, and $\pi:A\to\mathbb{Z}_n$ the associated power function. Then for each
$x\in A$,
\[
\sigma(x,m)\equiv0\pmod{m},
\]
 where $\sigma(x,m)=\sum\limits_{i=1}^{m-1}\pi(\varphi^{i-1}(x))$ and
$m=|O_x|$ is length of the orbit $O_x$ containing $x$. Moreover, $\sigma(x,n)\equiv0\pmod{n}$.
\end{lemma}
 \begin{proof}
 By Lemma~\ref{BASIC}(i) we have  $1=\varphi^m(xx^{-1})=\varphi^m(x)\varphi^{\sigma(x,m)}(x^{-1})=x\varphi^{\sigma(x,m)}(x^{-1})$, so
$\varphi^{\sigma(x,m)}(x^{-1})=x^{-1}$. By Lemma~\ref{INVERSE}, $m=|O_{x^{-1}}|$, so $\sigma(x,m)\equiv0\pmod{m}$. Since $m$ divides $n$, $\sigma(x,n)=\sum_{i=1}^n\pi(\varphi^{i-1}(x))=\frac{n}{m}\sigma(x,m)\equiv0\pmod{n}$.
 \end{proof}
\begin{lemma}\label{ORDER0}\cite{Hu2012}
  Let $\varphi$  be a skew-morphism of a finite group $A$, then for any $x,y\in A$, $|O_{xy}|$ divides $\lcm(|O_x|,|O_y|)$.
\end{lemma}
\begin{proof}
Denote $c=|O_x|$, $d=|O_y|$ and $l=\lcm(|O_x|,|O_y|)$. Then $l=cp=dq$ for some positive integers $p,q$.
By Lemma~\ref{BASIC}(i),  $\varphi^l(xy)=\varphi^l(x)\varphi^{\sigma(x,l)}(y)=x\varphi^{\sigma(x,l)}(y).$ We have
$\sigma(x,l)=\sum\limits_{i=1}^l\pi(\varphi^{i-1}(x))=p\sum\limits_{i=1}^{c}\pi(\varphi^{i-1}(x))=p\sigma(x,c)$. By Lemma~\ref{ORBIT1} $\sigma(x,c)\equiv0\pmod{c}$, so $\sigma(x,l)\equiv0\pmod{l}$. Hence $\varphi^l(xy)=xy$, which implies that $|O_{xy}|$ divides $l$.
\end{proof}
The first part of the following lemma was first proved in~\cite[Lemma~3.1]{ZD2016}\label{ORDER11}. For completeness we include a different proof.
\begin{lemma}
  Let $\varphi$ be a skew-morphism of a finite group $A$ of order $n$, let $\pi:A\to\mathbb{Z}_n$ be
the associated power function.  If $A=\langle x_1,\cdots,x_r\rangle$ then $n=\lcm(|O_{x_1}|,\cdots,|O_{x_r}|)$. Moreover,
for any $g\in A$, $\varphi(g)$ and $\pi(g)$ are completely determined by the action of $\varphi$ and $\pi$ on the generating
orbits $O_{x_1},\cdots,O_{x_r}$.
\end{lemma}
\begin{proof}
Since $A=\langle x_1,\cdots,x_r\rangle$, any element $g$ of $A$ can be expressed as a product of  finite length $k$ in
 the generators $x_1,\cdots, x_r$. By Lemma~\ref{ORDER0} and using induction on $k$
it can be easily proved that $|O_g|$ divides $\lcm(|O_{x_1}|,\cdots,|O_{x_r}|)$, whence $n=\lcm(|O_{x_1}|,\cdots,|O_{x_r}|)$.

Moreover, to prove the second part we use induction
on the length $k$ of $g$. If $k=1$ then $g$ is a generator of $A$, the assertion is clearly true.
Assume the assertion for words of length $k$. Then for a word $g$ of length $k+1$,
we have $g=hx$ where $h$ is a word of length $k$ in the generators and $x\in\{x_1,\cdots,x_r\}$.
Then by Lemma~\ref{BASIC}(i) and (ii), we have
\[
\varphi(g)=\varphi(hx)=\varphi(h)\varphi^{\pi(h)}(x)\quad\text{and}\quad
\pi(g)\equiv\pi(hx)\equiv\sum\limits_{i=1}^{\pi(h)}\pi(\varphi^{i-1}(x))\pmod{n}.
\]
Since $\varphi(h)$ and $\pi(h)$ are completely determined by the action of $\varphi$ and $\pi$
on the generating orbits, so are $\varphi(g)$ and $\pi(g)$, as required.
\end{proof}
\begin{lemma}\label{QUOTIENT}\cite[Lemma 3.3]{ZD2016}
  Let $\varphi$ be a skew-morphism of a finite group $A$ of order $n$, let $\pi:A\to\mathbb{Z}_n$ be
the associated power function. If $N$ is a $\varphi$-invariant normal subgroup of $A$, then
\begin{itemize}
\item[\rm(i)]$\varphi$ induces a skew-morphism $\bar\varphi$ of $\bar A=A/N$ by defining $\bar\varphi$ as $\bar\varphi(\bar x)=\overline{\varphi(x)}$
and the power function $\bar\pi:\bar A\to\mathbb{Z}_{m}$ associated with $\bar\varphi$ is determined
by $\bar\pi(\bar x)\equiv\pi(x)\pmod{m}$ where $m=|\bar\varphi|$,
\item[\rm(ii)]$\Ker\varphi N/N\leq\Ker\bar\varphi$, $\Core\varphi N/N\leq\Core\bar\varphi$ and $\Fix\varphi N/N\leq\Fix\bar\varphi$.
\end{itemize}
\end{lemma}
\begin{proof}
 The proof of (i) can be found in \cite[Lemma~3.3]{ZD2016} while (ii) is obvious.
\end{proof}

\section{Invariant subgroups}
Let $\varphi$ be a skew-morphism of a finite group $A$. A subset $N$ of $A$ will be called \textit{$\varphi$-invariant} if $\varphi(N)=N$. In
particular if $N$ is a subgroup of $A$ then it will be called a $\varphi$-invariant subgroup of $A$.
\begin{proposition}
Let $\varphi$ be a skew-morphism of a finite group $A$. If $M$ and $N$ are $\varphi$-invariant subsets of $A$, then so are $M\cap N$ and $MN$.
In particular, if $M$ and $N$ are $\varphi$-invariant normal subgroups of $A$, then so are $M\cap N$ and $MN$.
\end{proposition}
\begin{proof}
For any $y\in\varphi( M\cap N)$, there exists $x\in M\cap N$ such that $y=\varphi(x)$.
Since $M$ and $N$ are both $\varphi$-invariant, $\varphi(x)\in M$ and $\varphi(x)\in N$, so $y\in M\cap N$, whence
 $\varphi(M\cap N)=M\cap N$. Therefore $M\cap N$ is also $\varphi$-invariant. Similarly for any $y\in \varphi(MN)$, there exist
 $u\in M$ and $v\in N$ such that $y=\varphi(uv)$. We have
$y=\varphi(uv)=\varphi(u)\varphi^{\pi(u)}(v)\in \varphi(M)\varphi(N)=MN$, so $\varphi(MN)= MN$, whence $MN$ is also $\varphi$-invariant.
\end{proof}

Let $\Pi$ be a finite set of primes, a positive integer $k$ will be called a \textit{$\Pi$-number} if all
prime factors of $k$ belong to $\Pi$. For instance, if $\Pi=\{2,3\}$, then $2,6,9$ are $\Pi$-numbers,
whereas $5,10,30$ are not. We define $1$ to be a $\Pi$-number for any set $\Pi$ of primes. Moreover,
let $\varphi$ be a skew-morphism of $A$, an orbit of $\varphi$ will be called a \textit{$\Pi$-orbit} if its
length is a $\Pi$-number. Define $\Orbit^\Pi\varphi$ to be the union of all $\Pi$-orbits of $\varphi$, namely,
\[
\Orbit^\Pi\varphi=\{x\in A\mid \text{$|O_x|$ is a $\Pi$-number}\}.
\]
\begin{proposition}Let $\varphi$ be a skew-morphism of $A$, let $\Pi$ be a set of primes,
then $\Orbit^\Pi\varphi$ is a $\varphi$-invariant subgroup of $A$ containing $\Fix\varphi$.
\end{proposition}
\begin{proof}
By definition, all fixed points of $\varphi$ belong to $\Orbit^{\Pi}\varphi$, so $\Orbit^{\Pi}\varphi$ is not empty.
Moreover, for any $x,y\in \Orbit^{\Pi}\varphi$, $|O_x|$ and $|O_y|$ are $\Pi$-numbers, so $\lcm(|O_x|,|O_y|)$ is also a $\Pi$-number.
By Lemma~\ref{ORDER0}, $|O_{xy}|$ divides $\lcm(|O_x|,|O_y|)$. It follows that $|O_{xy}|$ is also a $\Pi$-number. Hence $xy\in \Orbit^{\Pi}\varphi$. Therefore, $\Orbit^{\Pi}\varphi$ is a subgroup of $A$, which is clearly $\varphi$-invariant.
\end{proof}
\begin{example}
Consider the skew-morphism of the cyclic group $\mathbb{Z}_{21}$ defined by
\[
\varphi=(0)(1,2,4,8,16,11)(3,6,12)(5,10,20,19,17,13)(7,14)(9,18,15).
\]
Then $\Orbit^{\{2\}}\varphi=\langle 7\rangle$, $\Orbit^{\{3\}}\varphi=\langle 3\rangle$,
$\Orbit^{\{5\}}\varphi=\langle 0\rangle$, and $\Orbit^{\{2,3\}}\varphi=\mathbb{Z}_{21}.$
\end{example}
In what follows we study $\varphi$-invariant subgroups via the covering of skew-morphisms.
\begin{definition}
Let $\varphi_i$ be skew-morphisms of finite groups $A_i$ $(i=1,2)$. If there is an epimorphism $\theta:A_1\to A_2$ such that for all $x\in A_1$
\[
\theta\varphi_1(x)=\varphi_2\theta(x),
\]
then $\varphi_1$ will be called a \textit{covering} (or a lift) of $\varphi_2$, and $\varphi_2$ will be called a \textit{projection} (or a quotient)
of $\varphi_1$. The covering will be denoted by $\varphi_1\to\varphi_2$, and the epimorphism $\theta:A_1\to A_2$ will be said to be associated with the covering.
\end{definition}
\begin{lemma}\label{COVER1}
Let $\varphi_i$ be skew-morphisms of finite groups $A_i$ $(i=1,2)$, let $\varphi_1\to\varphi_2$ be a covering between skew-morphisms,
and $\theta:A_1\to A_2$ the associated epimorphism. Then
\begin{itemize}
 \item[\rm(i)]every $\varphi_1$-invariant subgroup $M$ of $A_1$ projects to a $\varphi_2$-invariant subgroup $\theta(M)$ of $A_2$,
 \item[\rm(ii)]every $\varphi_2$-invariant subgroup $N$ of $A_2$ lifts to a $\varphi_1$-invariant subgroup $\theta^{-1}(N)$ of $A_1$.
\end{itemize}
\end{lemma}
\begin{proof}
(i)~For any $y\in\theta(M)$, $y=\theta(x)$ for some $x\in M$. Since $M$ is $\varphi_1$-invariant, $\varphi_1(x)\in M$,
so $\varphi_2(y)=\varphi_2\theta(x)=\theta\varphi_1(x)\in\theta(M)$, whence $\theta(M)$ is $\varphi_1$-invariant.

(ii)~For any $x\in\theta^{-1}(N)$, $y=\theta(x)\in N$. Since $N$ is $\varphi_2$-invariant, $\varphi_2(y)\in N$,
so $\theta\varphi_1(x)=\varphi_2\theta(x)=\varphi_2(y)\in N$. Hence $\varphi_1(x)\in\theta^{-1}(N)$.
\end{proof}
Since the identity subgroup, the fixed-point subgroup $\Fix\varphi_2$ and the core $\Core\varphi_2$ are all $\varphi_2$-invariant subgroups of $A_2$, by Lemma~\ref{COVER1}, the kernel $\Ker\theta=\theta^{-1}(1)$, the preimages $\theta^{-1}(\Fix\varphi_2)$ and $\theta^{-1}(\Core\varphi_2)$ are all $\varphi_1$-invariant subgroups of $A_1$. In particular, $\Ker\theta$ and $\theta^{-1}(\Core\varphi_2)$ are both normal in $A_1$.

Now we are ready to introduce a new invariant subgroup for an arbitrary skew-morphism $\varphi$. An element of $A$ will be called \textit{smooth} if
$\varphi(x)\equiv x\pmod{\Core\varphi}$. Define $\Smooth\varphi$ to be the set of smooth elements of $\varphi$ in $A$, that is,
\[
\Smooth\varphi=\{x\in A\mid \varphi(x)\equiv x\pmod{\Core\varphi}\}.
\]
\begin{proposition}\label{STABLE1}
Let $\varphi$ be a skew-morphism of a finite group $A$ of order $n$, and $\pi:A\to\mathbb{Z}_n$ be the associated power function.  Let $\bar{\varphi}$ be the induced skew-morphism of $\varphi$ on $\bar{A}=A/\Core\varphi$. Take arbitrary $x\in A$. Then the following are equivalent,
 \begin{itemize}
  \item[\rm(i)] $x\in\Smooth\varphi$,
\item[\rm(ii)] $\pi(\varphi^{i}(x))=\pi(x)$ for all nonnegative integers $i$,
 \item[\rm(iii)] $\bar{x}\in\Fix\bar{\varphi}$.
\end{itemize}
\end{proposition}
\begin{proof}
(i)$\Rightarrow$(ii). Since $x\in\Smooth\varphi$, $\varphi(x)=ux$ for some $u\in\Core\varphi$. It follows that $\varphi^{i}(x)=\varphi^{i-1}(u)\cdots\varphi(u)ux$ for all nonnegative integers $i$. Noting that $\varphi^{i-1}(u)\cdots\varphi(u)u\in\Core\varphi$, we have $\pi(\varphi^{i}(x))=\pi(x)$.

(ii)$\Rightarrow$(iii). Since $\pi(\varphi(x))=\pi(x)$, we have $\varphi(x)=ux$ for some $u\in\Ker\varphi$ and then $\varphi^{2}(x)=\varphi(ux)=\varphi(u)\varphi(x)=\varphi(u)ux$. Since $\pi(\varphi^{2}(x))=\pi(x)$, we get $\varphi(u)u\in\Ker\varphi$ and therefore $\varphi(u)\in\Ker\varphi$. Repeat the above process, we get $\varphi^{i}(u)\in\Ker\varphi$ for all positive integer $i$. It follows that $u\in\Core\varphi$ and then $\bar{\varphi}(\bar{x})=\bar{x}$, that is, $\bar{x}\in\Fix\bar{\varphi}$.

(iii)$\Rightarrow$(i). Since $\bar{x}\in\Fix\bar{\varphi}$, we have $\bar{\varphi}(\bar{x})=\bar{x}$ and then $\varphi(x)=ux$ for some $u\in\Core\varphi$. Therefore $x\in\Smooth\varphi$.
\end{proof}
The following proposition is a direct corollary of Proposition \ref{STABLE1} and the proof is omitted.
\begin{corollary} Suppose $\varphi$, $A$, $\bar\varphi$ and $\bar A$ are the same as Proposition \ref{STABLE1}. Then $\Smooth\varphi$ is a $\varphi$-invariant subgroup of $A$ and $\Fix\overline{\varphi}=\overline{\Smooth\varphi}$. In particular,
\begin{itemize}
\item[\rm(i)]$\Smooth\varphi=\Core\varphi$ if and only if $\Fix\bar\varphi=\bar1$,
\item[\rm(ii)]$\Smooth\varphi=A$ if and only if $\Fix\bar\varphi=\bar A$, and
\item[\rm(iii)]$\Smooth\varphi =\Fix\varphi$ if $\Core \varphi=1$.
\end{itemize}
\end{corollary}
\begin{example}\label{EXM2}
Consider a skew-morphism of the cyclic group $\mathbb{Z}_{18}$ defined by
\begin{align*}
 \varphi&=(0)(1,15,17,7,3,5,13,9,11)(2,14,8)(4,10,16)(6)(12),\\
\pi&=[\,1\,][2,~5,~8,~2,\,5,8,\,2,\,5,\,8\,]\,[\,7,\,7,\,7\,]\,[\,4,~4,~4]\,[\,1\,][\,1\,].
\end{align*}
Then $\Core\varphi=\langle 6\rangle$, so $\bar\varphi=(\bar0)(\bar1,\bar3,\bar5)(\bar2)(\bar4)$ and $\Smooth\varphi=\langle 2\rangle$.
\end{example}

\section{Smooth skew-morphisms}
In general the kernel $\Ker\varphi$ of a skew-morphism $\varphi$ does not have to be a $\varphi$-invariant subgroup. If  $\Ker\varphi$ is $\varphi$-invariant then $\varphi$ will be called \textit{kernel-preserving}.  Clearly, $\varphi$ is kernel-preserving if and only if
$\Core\varphi=\Ker\varphi$.

The following lemma summarizes some basic properties of kernel-preserving skew-morphisms.
\begin{lemma}\label{KERNELP}
Let $\varphi$ be a kernel-preserving skew-morphism of a finite group $A$ of order $n$, let $\pi:A\to\mathbb{Z}_n$ be the associated power function,
 then
\begin{itemize}
\item[\rm(i)]$K=\Ker\varphi$ is a normal subgroup of $A$, and the restriction of $\varphi$ to $K$ is an automorphism of $K$,
\item[\rm(ii)] for some positive integer $k$ if $\mu=\varphi^k$ is a skew-morphism of $A$, then $\Ker\varphi\leq\Ker\mu$,
\item[\rm(iii)]for any automorphism $\gamma$ of $A$, $\gamma^{-1}\varphi\gamma$ is a kernel-preserving skew-morphism of $A$,
\item[\rm(iv)]for any pair of elements $x\in A$ and $u\in \Ker\varphi$ there is a unique element $v\in \Ker\varphi$
such that $xu=vx$ and $\varphi(x)\varphi^{\pi(x)}(u)=\varphi(v)\varphi(x)$. In particular if $A$ is abelian then $\pi(x)\equiv1\pmod{k}$ where $k$ is the order of the restriction of $\varphi$ to $K$.
\end{itemize}
\end{lemma}
\begin{proof}(i)~Since $\varphi$ is kernel-preserving, $\Ker\varphi=\Core\varphi$, which is a normal subgroup of $A$.
Moreover, for all $x,y\in K$ we have $\varphi(xy)=\varphi(x)\varphi(y)$,  so $\varphi\restriction_K$ is an automorphism of $K$.

(ii)~Since $\varphi$ is kernel-preserving we have $\Ker\varphi=\Core\varphi$. By Lemma~\ref{POWER} $\Core\varphi\leq\Core\mu$. Since $\Core\mu\leq\Ker\mu$ we get $\Ker\varphi\leq\Ker\mu$.

(iii)~This is an immediate consequence of Lemma~\ref{CONJ}.

(iv)~Since $K\unlhd A$, for any pair $(x,u)$ of elements $x\in A$ and $u\in K$ there is a unique element $v\in K$ such that $xu=vx$.
Then $\varphi(x)\varphi^{\pi(x)}(u)=\varphi(xu)=\varphi(vx)=\varphi(v)\varphi(x).$ In particular if $A$ is abelian then $u=v$ and
$\varphi^{\pi(x)}(u)=\varphi(u)$ for all $u\in K$, so $\pi(x)\equiv1\pmod{k}$.
\end{proof}

It is well known that every skew-morphism of an abelian group is kernel-preserving~\cite[Lemma 5.1]{CJT2007}. For non-abelian simple groups we have
\begin{proposition}
Every kernel-preserving skew-morphism of a non-abelian finite simple group $A$ is an automorphism of $A$.
\end{proposition}
\begin{proof}
If $\varphi$ is not an automorphism of $A$ then by Lemma~\ref{LESS} we have $1<\Ker\varphi<A$. Since $\varphi$ is kernel-preserving, by Lemma~\ref{KERNELP}(i)  $\Ker\varphi\unlhd A$, a contradiction.
\end{proof}

 Let $\varphi$ be a skew-morphism of a finite group $A$. Recall that $\Smooth\varphi$ consists of elements $x\in A$ such that $\varphi(x)\equiv x\pmod{\Core\varphi}$. If $\Smooth\varphi=A$ then $\varphi$ will be called \textit{smooth}. The concept of smooth skew-morphisms was first introduced by Hu in the unpublished manuscript~\cite{Hu2012}. It was rediscovered by Bachrat\'y and Jajcay under the name of \textit{coset-preserving} skew-morphisms~\cite{BJ2014}.
\begin{lemma}Let $\varphi$ be a skew-morphism of a finite group $A$.  If $\varphi$ is smooth then every subgroup of $A$ containing $\Core\varphi$ is $\varphi$-invariant, and in particular, $\varphi$ is kernel-preserving.
\end{lemma}
\begin{proof}
By Proposition \ref{STABLE1} if $\varphi$ is smooth then the induced skew-morphism $\bar\varphi$ of $\bar A=A/\Core\varphi$ is the identity permutation on $\bar A$. Since every subgroup of $\bar A$ is $\bar\varphi$-invariant, it follows from Lemma~\ref{COVER1} that every subgroup of $A$ containing $\Core\varphi$ is $\varphi$-invariant. Since $\Core\varphi\leq\Ker\varphi$, $\varphi(\Ker\varphi)=\Ker\varphi$.
\end{proof}

The following lemma characterizes smooth skew-morphisms in terms of the power functions.
\begin{lemma}\label{PRESERVE}
Let $\varphi$ be a skew-morphism of a finite group $A$ of order $n$, let $\pi:A\to\mathbb{Z}_n$ be the associated power function.
 Then $\varphi$ is smooth if and only if $\pi(\varphi(x))=\pi(x)$ for all $x\in A$.
\end{lemma}
\begin{proof}If $\varphi$ is smooth then by Proposition~\ref{STABLE1},
$\pi(\varphi(x))=\pi(x)$ for all $x\in A$. Conversely, if $\pi(\varphi(x))=\pi(x)$ for all $x\in A$, then for all $u\in \Ker\varphi$ we have
$\pi(\varphi(u))=\pi(u)= 1$, so $\varphi(u)\in\Ker\varphi$, which implies that $\Ker\varphi=\Core\varphi$. Therefore by Lemma~\ref{BASIC}(iv) the condition $\pi(\varphi(x))=\pi(x)$ implies that $\varphi(x)\equiv x\pmod{\Core\varphi}$, that is, $\varphi$ is smooth.
\end{proof}
It turns out that any smooth skew-morphism $\varphi$ preserves cosets of $\Ker\varphi$ in $A$. It is this reason that smooth skew-morphisms
were also called coset-preserving skew-morphisms in \cite{BJ2014, BJ2016}.

For a skew-morphism $\varphi$ of $A$, the smallest positive integer $p$ such that $\pi(\varphi^p(x))=\pi(x)$ is called
the \textit{periodicity} of $\varphi$. Periodicity of skew-morphisms was introduced as a tool to study skew-morphisms of abelian groups~\cite{BJ2016}. The following theorem is a generalization of the results obtained in~\cite{BJ2016}.
\begin{theorem}\label{MAIN1}
Let $\varphi$ be a kernel-preserving skew-morphism of a finite group $A$ of order $n$ with $\pi$ being the associated power function, let $\bar\varphi$
be the induced skew-morphism of $\bar A=A/\Ker\varphi$ of order $m$ by $\varphi$, then
\begin{itemize}
\item[\rm(i)] $m$ is equal to the periodicity of $\varphi$, and in particular $m$ divides $n$,
\item[\rm(ii)] if $\varphi$ is non-trivial, then $\mu=\varphi^m$ is also non-trivial,
\item[\rm(iii)]$\mu=\varphi^m$ is a smooth skew-morphism of $A$ of order $n/m$, and in particular $\mu$ is
an automorphism of $A$ if and only if $\sigma(x,m)\equiv m\pmod{n}$ for all $x\in A$,
\item[\rm(iv)]$\bar \varphi$ is smooth if and only if $\pi(\varphi(x))\equiv \pi(x)\pmod{m}$ for all $x\in A$,
\item[\rm(v)]if $\bar x\in\Ker\bar\varphi$ then $\pi(x)\equiv1\pmod{m}$, and in
particular $\bar\varphi$ is an automorphism of $\bar A=A/K$ if and only if $\pi(x)\equiv1\pmod{m}$ for all $x\in A$.
\end{itemize}
\end{theorem}
\begin{proof}
(i)~Let $p$ be the periodicity of $\varphi$. Then for all $x\in A$ we have $\pi(\varphi^p(x))=\pi(x)$,
so $\varphi^p(x)=ux$ for some $u\in K:=\Ker\varphi$, or equivalently $\bar\varphi^p(\bar x)=\bar x$, which implies
that $m\leq p$. On the other hand,
since $|\bar\varphi|=m$, for any $x\in A$,  $\bar\varphi^m(\bar x)=\bar x$, so
there is an element $u\in K$ such that $\varphi^m(x)=ux$. Hence $\pi(\varphi^m(x))=\pi(ux)= \pi(x)$.
The minimality of $p$ then implies that $p\leq m$.

(ii)~If $\varphi$ is non-trivial, then $|A:\Ker\varphi|<|\varphi|=n$. By Lemma~\ref{LESS} $m=|\bar\varphi|\leq |\bar A|=|A:\Ker\varphi|$, so $m$ is a proper
divisor of $n$, whence $\varphi^m$ is non-trivial.

(iii)~By (i) for each $x\in A$ we have
\[
\sigma(x,n)=\sum_{i=1}^n\pi(\varphi^{i-1}(x))=\frac{n}{m}\sum_{i=1}^m\pi(\varphi^{i-1}(x))=\frac{n}{m}\sigma(x,m)\pmod{n}.
\]
By Lemma~\ref{ORBIT1} $\sigma(x,n)=0\pmod{n},$ so $\sigma(x,m)\equiv0\pmod{m}.$ Hence
 by Lemma~\ref{POWER} $\mu=\varphi^m$ is a skew-morphism of $A$ with its power function determined
 by $\pi_\mu(x)\equiv \sigma(x,m)/m\pmod{n/m}$. By (i), $\pi(\mu(x))=\pi(\varphi^m(x))=\pi(x)$,
so $\pi_\mu(\mu(x))=\pi_\mu(x)$ whence $\mu$ is smooth.

(iv)~ By Lemma \ref{PRESERVE}, $\bar\varphi$ is smooth if and only if $\bar\pi(\bar\varphi(\bar x))=\bar\pi(\bar x)$ for all $x\in A$,
 or equivalently $\pi(\varphi(x))\equiv\pi(x)\pmod{m}$.

 (v)~If $\bar x\in\Ker\bar\varphi$, then for all $y\in A$ we have
\[
\overline{\varphi(x)\varphi^{\pi(x)}(y)}=\overline{\varphi(xy)}=\bar\varphi(\bar x\bar y)
=\bar\varphi(\bar x)\bar\varphi(\bar y)=\overline{\varphi(x)\varphi(y)},
\]
so $\overline{\varphi^{\pi(x)}(y)}=\overline{\varphi(y)}$, and hence $\bar\varphi^{\pi(x)-1}(\bar y)=\bar y$.
Therefore $\pi(x)\equiv1\pmod{m}$.
\end{proof}
\begin{example}\label{EXM3}
 Consider a skew-morphism of the cyclic group $\mathbb{Z}_{18}$ given by
\begin{align*}
\varphi&=(0)(1,~5,13,11,7,17)(2,16,8,10,14,4)(3,5)(6,12)(9),\\
\pi&=[\,1\,][3,~5,~3~,5~,3,~5~][5,~3,~5,~3,~5,~3][\,1,1\,][1,~1]\,[\,1\,].
\end{align*}
Then $\Ker\varphi=\langle 3\rangle$ and $\bar\varphi=(\bar 0)(\bar 1,\bar 2).$
The periodicity of $\varphi$ is $2$, which is precisely the order of $\bar\varphi$. Since  $\sigma(x,2)\equiv0\pmod{2}$, by Theorem~\ref{MAIN1}(iii), $\mu=\varphi^2$
 is an automorphism of $A$.
\end{example}

The following theorem summarizes the most important properties of smooth skew-morphisms, see also \cite{BJ2014,Hu2012}.
\begin{theorem}\label{MAIN2}
 Let $\varphi$ be a smooth skew-morphism of $A$ of order $n$, let $\pi:A\to\mathbb{Z}_n$ be the associated power function. Then
\begin{itemize}
\item[\rm(i)] $\pi:A\to\mathbb{Z}_n$ is a group homomorphism from $A$ to the multiplicative
group $\mathbb{Z}_n^*$ with $\Ker\pi=\Ker\varphi$,
\item[\rm(ii)]for any $\varphi$-invariant normal subgroup $N$ of $A$, the induced skew-morphism $\bar\varphi$
on $A/N$ is also smooth, in particular, if $N=\Ker\varphi$ then $\bar\varphi$ is the identity permutation,
\item[\rm(iii)]for any positive integer $k$, $\mu=\varphi^k$ is a smooth skew-morphism,
\item[\rm(iv)]for any automorphism $\gamma$ of $A$, $\psi=\gamma^{-1}\varphi\gamma$ is a smooth skew-morphism of $A$.
\end{itemize}
\end{theorem}
\begin{proof}
By Proposition \ref{STABLE1}, $\pi(\varphi^i(x))=\pi(x)$ for all nonnegative integers $i$. Then by Lemma~\ref{BASIC}(ii)
$\pi(xy)\equiv\sum\limits_{i=1}^{\pi(x)}\pi(\varphi^{i-1}(y))\equiv\pi(x)\pi(y)\pmod n$. Therefore $\pi$ is a group homomorphism from $A$ to the multiplicative group $\mathbb{Z}_n^*$.

(ii)~Since $\varphi$ is smooth, we have $\pi(\varphi(x))=\pi(x)$ and then $\bar\pi(\bar\varphi(\bar x))=\bar\pi(\bar x))$ where $m=|\bar\varphi|$. By Lemma~\ref{PRESERVE}, $\bar\varphi$ is smooth.

(iii)~Recalling that $\pi(\varphi^i(x))=\pi(x)$ for all $i$,
we get \[\sigma(x,k)=\sum\limits_{i=1}^{k}\pi(\varphi^{i-1}(x))\equiv k\pi(x)\pmod n\] for any positive integer $k$, which implies the equation
$kt\equiv\sigma(x,k)\pmod{n}$ is solvable for all $x\in A$. Therefore by Lemma~\ref{POWER}
$\mu=\varphi^k$ is a skew-morphism of $A$ and the associated power function $\pi_\mu:A\to\mathbb{Z}_m$ is given by
$\pi_\mu(x)\equiv \pi(x)\pmod{m}$ where $m=n/\gcd(n,k)$ is the order of $\mu$. Since
$\pi_\mu(\mu(x))\equiv\pi(\varphi^k(x))\equiv\pi(x)\equiv\pi_\mu(x)\pmod{m}$,
by Lemma~\ref{PRESERVE} $\mu$ is also smooth.

(iv)~By Lemma~\ref{BASIC}(viii), $\psi=\gamma^{-1}\varphi\gamma$ is a skew-morphism
with $\Core\psi=\gamma^{-1}(\Core\varphi)$. Since $\varphi$ is smooth and $\gamma$ is an automorphism,
for all $x\in A$ we have $\varphi(\gamma(x))\equiv\gamma(x)\pmod{\Core\varphi}$, so
$\gamma^{-1}\varphi\gamma(x)\equiv x\pmod{\gamma^{-1}(\Core\varphi)}$,
that is, $\psi(x)\equiv x\pmod{\Core\psi}$. Therefore $\psi$ is also smooth.
\end{proof}

\section{Smooth skew-morphisms of the dihedral groups}
 Throughout this section $D_n$ will denote the dihedral group of order $2n$ given by the presentation
\begin{align}\label{DIH}
D_n=\langle a,b\mid a^n=b^2=1, b^{-1}ab=a^{-1}\rangle,\quad n\geq 3.
\end{align}

The following lemma determines the normal subgroups of $D_n$.
\begin{lemma}\label{DIHNORMAL}
Let $K$ be a proper normal subgroup of $D_n$, $n\geq 3$. Then
\begin{itemize}
\item[\rm(i)]if $n$ is odd then $K=\langle a^u\rangle$ where $u$ divides $n$,
\item[\rm(ii)]if $n$ is even then either $K=\langle a^2,b\rangle$, $K=\langle a^2,ab\rangle$
or $K=\langle a^u\rangle$ where $u$ divides $n$.
\end{itemize}
\end{lemma}
\begin{proof}
First note that all elements of $D_n$ can be written as the form $a^u$ or $a^vb$ where $0\leq u,v<n$.
If $K$ contains no elements of the form $a^vb$ then $K=\langle a^u\rangle$ for some $u$ dividing $n$.
It is clear that all such subgroups are normal in $D_n$.

On the other hand if $K$ contains an element $x=a^vb$, then since $\langle x\rangle\not\trianglelefteq D_n$,
 $K$ must contain another element $y\not\in \langle x\rangle$. If $y=a^ub$ then
$y=a^{u}b=(a^{v}b)^{-1}a^{v-u}=x^{-1}a^{v-u}.$ Hence without loss of generality we may assume that $y=a^s$
where $s$ is the smallest positive integer such that $a^s\in K$.

We proceed to show that $K=\langle a^s,a^vb\rangle$. For any $z\in K$, we have $z=a^k$ or $z=a^kb$ for some integer $k$.  If $z=a^k$, then by the division algorithm there are
two integers $q,r$ such that $k=sq+r$ where $0\leq r<s$. It follows that $a^r=a^{k-sq}=a^k(a^s)^{-q}\in K$.
By the minimality of $s$ we have $r=0$, so $a^k\in\langle a^s,a^vb\rangle$. On the other hand, if
$z=a^kb$ then $a^{k-v}=a^kba^vb\in K$, so as the former case we have $a^{k-v}\in \langle a^s,a^vb\rangle$. Hence
$z=a^{k}b=a^{k-v}(a^vb)^{-1}\in \langle a^s,a^vb\rangle$. Therefore $K=\langle a^s,a^vb\rangle$.

Since $K\unlhd D_n$, $[a,a^vb]\in K$. We have $[a,a^vb]=[a,b]=a^{-2}$, so $a^{-2}\in K$, whence
$\langle a^2,a^vb\rangle\leq K$. If $n$ is odd then
$\langle a\rangle=\langle a^2\rangle$, so $K=\langle a,a^vb\rangle=D_n$, contrast to the
assumption that $K<D_n$. Therefore $n$ is even. Since $\langle a^2,a^vb\rangle$ is a
maximal subgroup of $D_n$, we have $K=\langle a^2,a^vb\rangle$. If $v$ is even, then
$K=\langle a^2,a^vb\rangle=\langle a^2,b\rangle$. If $v=2v'+1$ is odd, then
$a^vb=a^{2v'+1}b=a^{2v'}(ab)\in\langle a^2,ab\rangle$, so
$K=\langle a^2,a^vb\rangle=\langle a^2,ab\rangle$, as claimed.
\end{proof}
\begin{lemma}\cite{CJT2016}\label{MAX}
 Let $\varphi$ be a skew-morphism of $D_n$ where $n\geq 3$, then $\Ker\varphi\neq\langle a\rangle$.
\end{lemma}

\begin{lemma}\label{CLASSLEM}
 Let $\varphi$ be a smooth skew-morphism of $D_n$, $n\geq 3$. If $n$ is odd,
then $\varphi$ is an automorphism of $A$, whereas if $n$ is even and $\varphi$
is not an automorphism of $D_n$ then $\Ker\varphi=\langle a^2\rangle$,
 $\Ker\varphi=\langle a^2,ab\rangle$ or $\Ker\varphi=\langle a^2,b\rangle$.
In particular, in the latter two cases $\varphi$ is a smooth skew-morphism with $\Ker\varphi=\langle a^2,b\rangle$
if and only if $\gamma^{-1}\varphi\gamma$ is a smooth skew-morphism with $\Ker\varphi=\langle a^2,ab\rangle$ where
$\gamma:a\mapsto a,b\mapsto ab$ is an automorphism of $D_n$.
\end{lemma}
\begin{proof}
 Assume that $\varphi$ is not an automorphism of $A$
then $1<\Ker\varphi<D_n$. Since $\varphi$ is smooth, by Theorem~\ref{MAIN2}(i) the power
function $\pi:D_n\to \mathbb{Z}_{|\varphi|}^*$ is a group homomorphism,
with $\Ker\pi=\Ker\varphi$. It follows that $\Ker\varphi$ is a nontrivial proper normal subgroup of $A$.
Since $\mathbb{Z}_{|\varphi|}^*$ is abelian, $D_n'\leq\Ker\varphi$ where $D_n'$ is the derived subgroup of $D_n$.

 If $n$ is odd then $D_n'=\langle a\rangle$ which is a
maximal subgroup of $D_n$. By Lemma~\ref{MAX} $\Ker\varphi\neq\langle a\rangle$, so $\Ker\varphi=D_n$, and hence
$\varphi$ is automorphism of $D_n$, a contradiction.

On the other hand if $n$ is even then $D_n'=\langle a^2\rangle$, so $\langle a^2\rangle\leq\Ker\varphi$. By Lemma~\ref{DIHNORMAL}
$\Ker\varphi\leq \langle a\rangle$, or
 $\Ker\varphi=\langle a^2,b\rangle$, or $\Ker\varphi=\langle a^2,ab\rangle$. For the first case,
 by lemma~\ref{MAX} $\Ker\varphi\neq\langle a\rangle$,
so $\Ker\varphi=\langle a^2\rangle$. For the latter two cases every smooth skew-morphism with kernel $\langle a,b\rangle$
 is conjugate to a skew-morphism with kernel $\langle a,ab\rangle$ by the automorphism $\gamma:a\mapsto a, b\mapsto ab$, as claimed.
\end{proof}
The following result classifies smooth skew-morphisms of the dihedral groups $D_n$ with $\Ker\varphi=\langle a^2\rangle$ for even integer $n\geq 4$.
\begin{theorem}\label{CLASS1}
Let $D_n$ be the dihedral group of order $2n$ where $n\geq 4$ is an even number.
Then every smooth skew-morphism of $D_n$ with $\Ker\varphi=\langle a^2\rangle$ is
 defined by
\begin{align}\label{SKEW1}
\begin{cases}
\varphi(a^{2i})=a^{2iu},\\
\varphi(a^{2i+1})=a^{2iu+2r+1},\\
\varphi(a^{2i}b)=a^{2iu+2s},\\
\varphi(a^{2i+1}b)=a^{2iu+2r+2s\sigma(u,e)+1}b
\end{cases}
\quad\text{and}\quad
\begin{cases}
\pi(a^{2i})=1,\\
\pi(a^{2i+1})=e,\\
\pi(a^{2i}b)=f,\\
\pi(a^{2i+1}b)=ef,
\end{cases}
\end{align}
where $r,s,u,k,e,f$ are nonnegative integers satisfying the following conditions
\begin{itemize}
 \item[\rm(i)]$r,s\in\mathbb{Z}_{n/2}$ and $u\in\mathbb{Z}_{n/2}^*$,
 \item[\rm(ii)]$k$ is the order of $\varphi$, which is the smallest positive integer such that
$r\sigma(u,k)\equiv0\pmod{n/2}$ and $s\sigma(u,k)\equiv0\pmod{n/2}$  where $\sigma(u,k)=\sum\limits_{i=1}^ku^{i-1},$
 \item[\rm(iii)] $e,f\in\mathbb{Z}_k^*$ such that $e\not\equiv1\pmod{k}$, $f\not\equiv1\pmod{k}$, $ef\not\equiv1\pmod{k}$, $e^2\equiv1\pmod{k}$
and $f^2\equiv1\pmod{k}$,
\item[\rm(iv)] $u^{e-1}\equiv1\pmod{n/2}$  and $u^{f-1}\equiv1\pmod{n/2}$,
 \item[\rm(v)]$r\sigma(u,e-1)\equiv u-2r-1\pmod{n/2}$,
\item[\rm(vi)]$s\sigma(u,f-1)\equiv 0\pmod{n/2}$,
\item[\rm(vii)]$r\sigma(u,f-1)+s\sigma(u,e-1)\equiv u-2r-1\pmod{n/2}$.
\end{itemize}
\end{theorem}
\begin{proof}
 By Theorem~\ref{MAIN2}, the induced skew-morphism $\bar\varphi$ on $D_n/\Ker\varphi$ is the identity permutation, so
  there exist integers $r,s\in\mathbb{Z}_{n/2}$ such that
\[
\varphi(a)=a^{1+2r}\quad\text{and}\quad \varphi(b)=a^{2s}b.
\]
Since $\varphi$ is kernel-preserving, the restriction of $\varphi$ to $\Ker\varphi=\langle a^2\rangle$ is an automorphism, so
$\varphi(a^2)=a^{2u}$ where $u\in\mathbb{Z}_{n/2}^*$. Assume that $\pi(a)\equiv e\pmod{k}$ and $\pi(b)\equiv f\pmod{k}$
where $k=|\varphi|$.

Using induction it is easy to show that for any positive integer $j$,
 \[
 \varphi^j(a)=a^{1+2r\sigma(u,j)},\quad  \varphi^j(b)=a^{2s\sigma(u,j)}b
 \]
where
\[
\sigma(u,j)=\sum_{i=1}^ju^{i-1}.
\]
Since $D_n=\langle a,b\rangle$, the order $k=|\varphi|$ is equal to $\lcm(|O_a|,|O_b|)$, the least common
multiple of the lengths of the orbits containing $a$ and $b$, or equivalently the smallest positive integer $k$
 such that $\varphi^k(a)=a$ and $\varphi^k(b)=b$. Using the above formula we then deduce that $k$ is
the smallest positive integer such that $r\sigma(u,k)\equiv0\pmod{n/2}$ and $s\sigma(u,k)\equiv0\pmod{n/2}.$

Now we determine the skew-morphism and the associated power function. By the assumption we have
 $\varphi(a^{2i})=\varphi((a^2)^{i})=(a^{2u})^{i}=a^{2iu}$ and
$\varphi(a^{2i}b)=\varphi(a^{2i})\varphi(b)=a^{2iu+2s}b$. Similarly,
$\varphi(a^{2i+1})=\varphi(a^{2i}a)=\varphi(a^{2i})\varphi(a)=a^{1+2r+2iu}$
and $\varphi(a^{2i+1}b)=\varphi(a^{2i})\varphi(a)\varphi^e(b)=a^{2iu+1+2r+2s\sigma(u,e)}$.
 Since $\pi:D_n\to\mathbb{Z}_k^*$ is a group homomorphism, we have $e^2\equiv \pi(a)^2=\pi(a^2)\equiv 1\pmod{k}$ and $e^2\equiv\pi(a)^2\equiv\pi(a^2)\equiv1\pmod{k},$  so $e^2\equiv1\pmod{k}$ and $f^2\equiv1\pmod{k}$. Hence $\pi(a^{2i})\equiv1$, $\pi(a^{2i+1})\equiv e$, $\pi(a^{2i}b)\equiv f$,
$\pi(a^{2i+1}b)\equiv ef$. In particular, since $|D_n:\Ker\varphi|=4$ is equal to the number of distinct
values of the power function, we have $e\not\equiv f\pmod{k}$, $e\not\equiv 1\pmod{k}$ and $f\not\equiv 1\pmod{k}$.
Therefore $\varphi$ and $\pi$ have the form given by Eq.~\eqref{SKEW1}.

Moreover, since $\varphi(a)\varphi^{e}(a^2)=\varphi(a)\varphi^{\pi(a)}(a^2)=\varphi(aa^2)=\varphi(a^{2}a)
=\varphi(a^2)\varphi(a),$ we get $a^{1+2r+2u^e}=\varphi(a)\varphi^{e}(a^2)=\varphi(a^2)\varphi(a)=a^{1+2r+2u}.$
Hence $u^{e-1}\equiv1\pmod{n/2}$. Similarly, since $\varphi(b)\varphi^{f}(a^2)=\varphi(b)\varphi^{\pi(b)}(a^2)=\varphi(ba^2)
=\varphi(a^{-2}b)=\varphi(a^{-2})\varphi(b),$ we obtain $a^{2s-2u^f}b=a^{2s}ba^{2u^f}=\varphi(b)\varphi^{f}(a^2)=\varphi(a^{-2})\varphi(b)=a^{2s-2u}b.$
Hence $u^{f-1}\equiv1\pmod{n/2}$.

Furthermore,  since $a^{2u}=\varphi(a^2)=\varphi(a)\varphi^{\pi(a)}(a)=\varphi(a)\varphi^e(a)=a^{2+2r+2r\sigma(u,e)},$ we get
\begin{align}\label{SEQ1}
r(1+\sigma(u,e))\equiv u-1\pmod{n/2}.
\end{align}
Similarly $1=\varphi(b^2)=\varphi(b)\varphi^{\pi(b)}(b)
=\varphi(b)\varphi^{f}(b)=a^{2s}ba^{2s\sigma(u,f)}b=a^{2s-2s\sigma(u,f)},$ we obtain
\begin{align}\label{SEQ2}
s\sigma(u,f)\equiv s\pmod{n/2}.
\end{align}

Employing induction it is easy to derive $\varphi^j(a^{-1})=a^{1-2u^j+2r\sigma(u,j)}$ where $j$ is an arbitrary positive integer. Then $\varphi(a)\varphi^e(b)=\varphi(ab)=\varphi(ba^{-1})=\varphi(b)\varphi^f(a^{-1})$. Upon substitution we get $a^{1+2r+2s\sigma(u,e)}b
=\varphi(a)\varphi^e(b)=\varphi(b)\varphi^f(a^{-1})=a^{2s}ba^{1-2u^f+2r\sigma(u,f)}=a^{2s-1+2u^f-2r\sigma(u,f)}b.$ Hence
$r\sigma(u,f)+s\sigma(u,e)\equiv s+u^f-r-1\pmod{n/2}.$ Since $u^f\equiv u\pmod{n/2}$ the congruence is reduced to
\begin{align}\label{SEQ3}
r\sigma(u,f)+s\sigma(u,e)\equiv s+u-r-1\pmod{n/2}.
\end{align}
Recall that $u^{e-1}\equiv 1\pmod{n/2}$ and $u^{f-1}\equiv1\pmod{n/2}$, so $\sigma(u,e)\equiv\sigma(u,e-1)+1\pmod{n/2}$
and $\sigma(u,f)\equiv \sigma(u,f-1)+1\pmod{n/2}$. Upon substitution the congruences \eqref{SEQ1}, \eqref{SEQ2} and \eqref{SEQ3} are
reduced to (v), (vi) and (vii), respectively.

Conversely, for any quintuple $(r,s,u,e,f)$ of nonnegative integers
satisfying the stated numerical conditions, it is straightforward
to verify that $\varphi$ given by Eqn.~\eqref{SKEW1}
is a smooth skew-morphism of $D_n$ of order $k$ with
$\Ker\varphi=\langle a^2\rangle$ and the function $\pi$  is the associated power function. We leave it as an exercise to the reader.
\end{proof}

\begin{remark}In Theorem~\ref{CLASS1}, consider the particular case where $u=1$. By (ii) we have
\[
k=\lcm(\frac{n/2}{\gcd(r,n/2)},\frac{n/2}{\gcd(s,n/2)}).
\]
The numerical conditions are reduced to
\[
\begin{cases}
 e^2\equiv1\pmod{k},\\
f^2\equiv1\pmod{k},\\
r(e+1)\equiv0\pmod{n/2},\\
s(f-1)\equiv0\pmod{n/2},\\
r(f+1)+s(e-1)\equiv0\pmod{n/2},
\end{cases}
\]
where $r,s\in\mathbb{Z}_{n/2}$, $e,f\in\mathbb{Z}_{k}$ such that $e\not\equiv1\pmod{k}$, $f\not\equiv1\pmod{k}$ and $ef\not\equiv1\pmod{k}$.
 If $n=8m$ where $m\geq 3$ is an odd number, then it can be easily verified that the quintuple $(r,s,u,e,f)=(m+4,m, 1, 4m-1, 2m-1)$
fulfil the numerical conditions. Therefore we obtain an infinite family of skew-morphisms of $D_{8m}$ of order $4m$
with $\Ker\varphi=\langle a^2\rangle$. This example
was first discovered by Zhang and Du in \cite[Example 1.4]{ZD2016}.
\end{remark}

The following theorem classifies smooth skew-morphisms of the dihedral group $D_n$ with $\Ker\varphi=\langle a^2,b\rangle$ where $n\geq8$ is even.
\begin{theorem}\label{CLASS2}
 Let $D_n$ be the dihedral group of order $2n$ where $n\geq 8$ is an even number. If $\varphi$ is a
smooth skew-morphism of $D_n$ with $\Ker\varphi=\langle a^2,b\rangle$ then $\varphi$ belongs to one
of the following two families of skew-morphisms:
\begin{itemize}
 \item[\rm(I)]skew-morphisms of order $k$ defined by
\begin{align}\label{SKEW2}
 \begin{cases}
 \varphi(a^{2i})=a^{2iu},\\
\varphi(a^{2i+1})=a^{(2i+1)u+2r+1},\\
\varphi(ba^{2i})=ba^{2iu+2s},\\
\varphi(ba^{2i+1})=ba^{2r+2s+2iu+1}
\end{cases}
\quad\text{and}\quad
\begin{cases}
 \pi(a^{2i})=1,\\
\pi(a^{2i+1})=e,\\
\pi(ba^{2i})=1,\\
\pi(ba^{2i+1})=e,
\end{cases}
\end{align}
where $r,s,u,k,e$ are nonnegative integers satisfying the following conditions
\begin{itemize}
\item[\rm(i)]$r,s\in\mathbb{Z}_{n/2}$, $u\in\mathbb{Z}_{n/2}^*$ such that $u-1-2r\not\equiv0\pmod{n/2}$,
\item[\rm(ii)] $k$ is the smallest positive integer such that $r\sigma(u,k)\equiv0\pmod{n/2}$
and $s\sigma(u,k)\equiv0\pmod{n/2}$ where $\sigma(u,k)=\sum\limits_{i=1}^ku^{i-1}$,
\item[\rm(iii)]$e\in\mathbb{Z}_k^*$ such that $e\not\equiv1\pmod{k}$, $e^2\equiv1\pmod{k}$ and $u^{e-1}\equiv1\pmod{n/2}$,
 \item[\rm(iv)]$r\sigma(u,e-1)\equiv u-2r-1\pmod{n/2}$,
\item[\rm(v)] $s\sigma(u,e-1)\equiv -u+2r+1\pmod{n/2}.$
\end{itemize}
\item[\rm(II)]skew-morphisms of order $2(e-1)$ defined by
\begin{align}\label{SKEW3}
  \begin{cases}
\varphi(a^{2i})=a^{2iu},\\
\varphi(a^{2i+1})=ba^{2r-2iu+1},\\
\varphi(ba^{2i})=ba^{2s+2iu},\\
\varphi(ba^{2i+1})=a^{2r-2s-2iu+1}
\end{cases}
\quad\text{and}\quad
\begin{cases}
 \pi(a^{2i})=1,\\
\pi(a^{2i+1})=e,\\
\pi(ba^{2i})=1,\\
\pi(ba^{2i+1})=e,
\end{cases}
\end{align}
where $r,s,u,e$ are nonnegative integers satisfying the following conditions
\begin{itemize}
\item[\rm(i)]$r,s\in\mathbb{Z}_{n/2}$, $u\in\mathbb{Z}_{n/2}^*$ and $e>1$ is an odd number,
 \item[\rm(ii)]$u^{e-1}\equiv-1\pmod{n/2}$,
 \item[\rm(iii)]$s\sigma(u,e-1)\equiv u+2r+1\pmod{n/2}$ where $\sigma(u,e-1)=\sum\limits_{i=1}^{e-1}u^{i-1}$,
\item[\rm(iv)]$r\xi(u,e-1)\equiv s\zeta(u,e-1)-1\pmod{n/2}$ where $\xi(u,e-1)=\sum\limits_{i=1}^{e-1}(-u)^{i-1}$ and
$\zeta(u,e-1)=\sum\limits_{i=1}^{(e-1)/2}u^{2(i-1)}$.
\end{itemize}
\end{itemize}
\end{theorem}
\begin{proof}
By Theorem~\ref{MAIN2} the induced skew-morphism $\bar\varphi$ of $D_n/\Ker\varphi$ is the identity
and the restriction of $\varphi$ to $\Ker\varphi=\langle a^2,b\rangle$ is an
automorphism of $\Ker\varphi$. It follows that there exist integers $r,s,u\in\mathbb{Z}_{n/2}$
and $l\in\mathbb{Z}_2$ such that
\[
\varphi(a)=b^la^{1+2r},\quad \varphi(b)=ba^{2s}\quad\text{and}\quad \varphi(a^2)=a^{2u}.
\]
Assume that $\pi(a)\equiv e\pmod{k}$ where $k=|\varphi|$ denotes the order of $\varphi$.
 Since $b\in\Ker\varphi$, $\pi(b)\equiv1\pmod{k}$ . By Theorem~\ref{MAIN2} the power function
$\pi:D_n\to\mathbb{Z}_k$ is a group homomorphism from $D_n$ to the multiplicative
group $\mathbb{Z}_k^*$, so
\[
e^{-1}\equiv\pi(a^{-1})\equiv\pi(b^{-1}ab)\equiv\pi(b)\equiv  e\pmod{k},
\]
which is equivalent to $e^2\equiv1\pmod{k}$. Hence $\pi(a^{2i})\equiv\pi(a^{2i}b)\equiv1$ and $\pi(a^{2i+1})\equiv\pi(a^{2i+1}b)\equiv e$.
 Since the number of distinct values of the power
function is equal to  $|D_n:\Ker\varphi|=2$, we have $e\not\equiv1\pmod{k}$. To proceed we distinguish two cases:

\begin{case}[I]~$l=0$.\par
In this case we have
\[
\varphi(a)=a^{1+2r},\quad \varphi(b)=ba^{2s}\quad\text{and}\quad \varphi(a^2)=a^{2u}.
\]
Then $\varphi(a^{2i})=\varphi(a^2)^{i}=a^{2iu}$ and
$\varphi(ba^{2i})=\varphi(b)\varphi(a^2)^{i}=ba^{2iu+2s}$. Similarly, $\varphi(a^{2i+1})=\varphi(a^{2i}a)=\varphi(a^2)^{i}\varphi(a)=
a^{2iu+2r+1}$ and $\varphi(ba^{2i+1})=\varphi(ba^{2i}a)=\varphi(b)\varphi(a^{2i})\varphi(a)
=ba^{2r+2s+2iu+1}.$ Hence the skew-morphism has the form given by Eq.~\eqref{SKEW2}.

Using induction it is easy to prove that
$\varphi^j(a)=a^{1+2r\sigma(u,j)}$ and $\varphi^j(b)=ba^{2s\sigma(u,j)}$ where $j$ is a positive integer and
$\sigma(u,j)=\sum\limits_{i=1}^ju^{i-1}.$ Since $D_n=\langle a,b\rangle$, $k=|\varphi|$ is the smallest positive integer
such that $\varphi^k(a)=a$ and $\varphi^k(b)=b$, which imply that $r\sigma(u,k)\equiv0\pmod{n/2}$ and
$s\sigma(u,k)\equiv0\pmod{n/2}.$

Moreover,  since $\varphi(a)\varphi^e(a^2)=\varphi(aa^2)=\varphi(a^2a)=\varphi(a^2)\varphi(a)$,
we have $\varphi(a)\varphi^e(a^2)=a^{1+2r+2u^e}$ and $\varphi(a^2)\varphi(a)=a^{1+2r+2u},$ so $u^{e-1}\equiv1\pmod{n/2}$.

Note that $a^{2u}=\varphi(a^2)=\varphi(a)\varphi^e(a)=a^{1+2r}a^{1+2r\sigma(u,e)}=a^{2+2r+2r\sigma(u,e)},$
so we obtain
\begin{align}\label{EQTR1}
 r\big(\sigma(u,e)+1\big)\equiv u-1\pmod{n/2}.
\end{align}
Similarly, $\varphi(a)\varphi^e(b)=\varphi(ab)=\varphi(ba^{-1})=\varphi(b)\varphi(a^{-1})
=\varphi(b)\varphi(a^{-2}a)=\varphi(b)\varphi(a^{-2})\varphi(a).$ By the above formula
$\varphi(a)\varphi^e(b)=a^{1+2r}ba^{2s\sigma(u,e)}=ba^{-1-2r+2s\sigma(u,e)}$ and
$\varphi(b)\varphi(a^{-2})\varphi(a)=ba^{1+2r+2s-2u}.$ Consequently
\begin{align}\label{EQTR2}
s(\sigma(u,e)-1)\equiv-u+2r+1\pmod{n/2}.
\end{align}
Recall that $u^{e-1}\equiv1\pmod{n/2}$, so $\sigma(u,e)=\sigma(u,e-1)+u^{e-1}\equiv\sigma(u,e-1)+1\pmod{n/2}$.
Upon substitution Eqs.~\eqref{EQTR1} and \eqref{EQTR2} are reduced to $r\sigma(u,e-1)\equiv u-2r-1\pmod{n/2}$ and
$s\sigma(u,e-1)\equiv-u+2r+1\pmod{n/2}$. Since $e<k$, the minimality of $k$ implies that
$r\sigma(u,e-1)\not\equiv0\pmod{n/2}$ or $s\sigma(u,e-1)\not\equiv0\pmod{n/2}$, so $u-2r-1\not\equiv0\pmod{n/2}$.

\end{case}

\begin{case}[II] $l=1$.\par
In this case we have
\[
\varphi(a)=ba^{1+2r},\quad \varphi(b)=ba^{2s}\quad\text{ and }\quad\varphi(a^2)=a^{2u}.
\]
Then $\varphi(a^{2i})=a^{2iu}$ and $\varphi(ba^{2i})=\varphi(b)\varphi(a^{2i})=ba^{2s+2iu}$. Similarly,
$\varphi(a^{2i+1})=\varphi(a^{2i}a)=a^{2iu}ba^{1+2r}=ba^{2r-2iu+1}$ and $\varphi(ba^{2i+1})= a^{2r-2s-2iu+1}$.
Hence $\varphi$ has the form given by Eq.~\eqref{SKEW3}.

Using induction it is easy to derive the following formula
\[
\varphi^j(b)=ba^{2s\sigma(u,j)} \quad\text{and}\quad
\varphi^j(a)=
\begin{cases}
a^{2r\xi(u,j)-2s\zeta(u,j)+1},&\text{if $j$ is even,}\\
ba^{2r\xi(u,j)+2su\zeta(u,j-1)+1},&\text{if $j$ is odd}
\end{cases}
\]
where $j\geq2$ is a positive integer and
\[
 \sigma(u,j)=\sum_{i=1}^ju^{i-1},\quad \xi(u,j)=\sum_{i=1}^j(-u)^{j-1}\quad\text{and}\quad\zeta(u,j)=\sum_{i=1}^{j/2}u^{2(i-1)}.
\]
Since $\varphi(a)=ba^{1+2r}$ and $D_n=\langle a,ba^{1+2r}\rangle$, $k=|\varphi|=|O_a|$, so $k$ is the smallest positive integer such that
 $r\xi(u,k)\equiv s\zeta(u,k)\pmod{n/2}$. In particular we see $k$ must be even.

Note that $\varphi(a^2)=\varphi(a)\varphi^e(a)$. Since $\gcd(e,k)=1$, $e$ is odd, so by the above formula
$\varphi(a)\varphi^e(a)=ba^{1+2r}ba^{2r\xi(u,e)+2st\zeta(u,e-1)+1} =a^{2r\xi(u,e)-2r+2su\zeta(u,e-1)}.$
Recall that $\varphi(a^2)=a^{2u}$. Consequently we obtain
\begin{align}\label{CAYEQ1}
r\xi(u,e)+su\zeta(u,e-1)\equiv r+u\pmod{n/2}.
\end{align}

Furthermore, $\varphi(a)\varphi^e(a^2)=\varphi(aa^2)=\varphi(a^2a)=\varphi(a^2)\varphi(a)$. By the above formula we have
$\varphi(a)\varphi^e(a^2)=ba^{1+2r+2u^e}$ and $\varphi(a^2)\varphi(a)=a^{2u}ba^{1+2r}=ba^{2r-2u+1},$
so $u^{e-1}\equiv-1\pmod{n/2}$. Similarly $\varphi(a)\varphi^e(b)=\varphi(ab)
=\varphi(ba^{-2}a)=\varphi(b)\varphi(a^{-2})\varphi(a),$ using substitution we get
$\varphi(a)\varphi^e(b)=ba^{1+2r}ba^{2s\sigma(u,e)}=a^{-1-2r+2s\sigma(u,e)}$ and
$\varphi(b)\varphi(a^{-2})\varphi(a)=ba^{2s-2u}ba^{1+2r}=a^{1+2r-2s+2u}.$  Hence
\begin{align}\label{CAYEQ2}
s\sigma(u,e)\equiv 1+2r+u-s\pmod{n/2}.
\end{align}
Recall that $u^{e-1}\equiv-1\pmod{n/2}$, so $\sigma(u,e)\equiv\sigma(u,e-1)-1\pmod{n/2}$
and $\xi(u,e)\equiv\xi(u,e-1)-1\pmod{n/2}$. Upon substitution Eqs.~\eqref{CAYEQ1} and \eqref{CAYEQ2} are
reduced to
\begin{align}
&r\xi(u,e-1)+su\zeta(u,e-1)\equiv 2r+u\pmod{n/2},\\
&s\sigma(u,e-1)\equiv2r+u+1\pmod{n/2}.
\end{align}
Subtracting we get $r\xi(u,e-1)\equiv s\zeta(u,e-1)-1\pmod{n/2}$.

Finally, note that
\[
\xi(u,2(e-1))=\sum_{i=1}^{2(e-1)}(-u)^{2(e-1)}=\sum_{i=1}^{e-1}(-u)^{i-1}+u^{e-1}\sum_{i=1}^{e-1}(-u)^{i-1}\equiv0\pmod{n/2},
\]
and
\[
\zeta(u,2(e-1))=\sum_{i=1}^{e-1}u^{2i}=\sum_{i=1}^{(e-1)/2}u^{2(i-1)}+u^{e-1}\sum_{i=1}^{(e-1)/2}u^{2(i-1)}\equiv0\pmod{n/2},
\]
Hence $r\xi(u,2(e-1))\equiv s\zeta(u,2(e-1))\pmod{n/2}$. The minimality of $k$ yields $k\mid 2(e-1)$. But $e-1<k$, which forces
$k=2(e-1)$.
\end{case}
Conversely, in each case for any quadruple $(r,s,u,e)$ satisfying the numerical conditions, it is straightforward to
verify that $\varphi$ of the given form is a smooth skew-morphism of $D_n$ with $\Ker\varphi=\langle a^2,b\rangle$ and $\pi$ is the
associated power function. The details are left to the reader.
\end{proof}
\begin{remark}Let $\varphi$ be a skew-morphism from (II) of Theorem~\ref{CLASS1}. Note that the orbit of $\varphi$
 containing $a^{2i+1}$ also contains $ba^{2r-2i+1}$, so the orbit $O_a$ generates $D_n$. Clearly $O_a$ is closed under taking inverse. Therefore
$\varphi$ gives rise to an $e$-balanced regular Cayley map of $D_n$ of even valency $2(e-1)$. Such Cayley maps were first
classified by Kwak, Kwon and Feng in \cite{KKF2006}.
\end{remark}

\section*{Acknowledgement}
The first and second author are supported by the following grants: Natural Science Foundation of Zhejiang Province (LQ17A010003, LY16A010010) and Scientific Research Foundation of Zhejiang Ocean University (21065014115, 21065014015). The third and fourth author are supported by the National Natural Science Foundation of China (11671276). The fourth author is supported by National Natural Science Foundation of China (11401290) and Natural Science Foundation of Fujian (2016J01027).

\end{document}